\documentclass[12pt, reqno]{amsart}
\usepackage[table]{xcolor}
\usepackage{amsmath,amssymb, amsfonts,amscd,psfrag,graphicx,stmaryrd}
\usepackage{epsfig}
\usepackage{amsthm}
\usepackage{fullpage}
\usepackage{verbatim}
\usepackage{listings}
\usepackage[table]{xcolor} 
\usepackage{mathtools}
\DeclarePairedDelimiter{\ceil}{\lceil}{\rceil}
\DeclarePairedDelimiter\floor{\lfloor}{\rfloor}
\newcommand{\lebn}

\usepackage[misc]{ifsym}
\usepackage{caption}
\usepackage{float}
\usepackage{amsmath}
\usepackage{subfigure}

\theoremstyle{plain}
\newtheorem{prop}[equation]{Proposition}

\newtheorem{thm}[equation]{Theorem}
\newtheorem{conj}[equation]{Conjecture}

\newtheorem{cor}[equation]{Corollary}
\newtheorem{lem}[equation]{Lemma}

\theoremstyle{definition}

\numberwithin{equation}{section}

\newcommand{\D}{\Delta}

\usepackage[a4paper,left=1.5cm,right=1.5cm,top=2cm,bottom=2cm]{geometry}

\allowdisplaybreaks

\usepackage{tikz}
\usepackage{verbatim}
\usetikzlibrary{shapes,shadows,calc}
\usepgflibrary{arrows}
\usetikzlibrary{arrows, decorations.markings, calc, fadings, decorations.pathreplacing, patterns, decorations.pathmorphing, positioning}
\tikzset{nodc/.style={circle,draw=blue!50,fill=pink!80,inner sep=1.6pt}}
\tikzset{nodr/.style={circle,draw=black,fill=green!50!black,inner sep=2pt}}
\tikzset{nodel/.style={circle,draw=black,inner sep=2.2pt}}
\tikzset{nodinvisible/.style={circle,draw=white,inner sep=2pt}}
\tikzset{nodpale/.style={circle,draw=gray,fill=gray,inner sep=1.6pt}}

\tikzset{nod1/.style={circle,draw=black,fill=black,inner sep=1pt}}
\tikzset{nod2/.style={circle,draw=black,fill=blue!75!black,inner sep=1.6pt}}
\tikzset{nod3/.style={circle,draw=black,fill=black,inner sep=1.8pt}}
\tikzset{noddiam/.style={diamond,draw=black,inner sep=2pt}}
\tikzset{nodw/.style={circle,draw=black,inner sep=1.8pt}}

\usepackage[colorlinks, urlcolor=blue, linkcolor=blue, citecolor=blue]{hyperref}

\makeatletter
 \def\@textbottom{\vskip \z@ \@plus 10pt}
 \let\@texttop\relax
\makeatother

\usepackage[utf8]{inputenc}
\usepackage[T1]{fontenc}

\begin{document}

\bibliographystyle{plain}

\title[2-distance coloring of planar graphs with girth five]{On 2-Distance ($\D+4$)-coloring of planar graphs with girth at least five}  

\author{Zakir Deniz}
\thanks{The author is supported by T\" UB\. ITAK, grant no:122F250}
\address{Department of Mathematics, Duzce University, Duzce, 81620, Turkey.}
\email{zakirdeniz@duzce.edu.tr}

\date{\today}
\thanks{}

\begin{abstract}
A vertex coloring of a graph $G$ is called a 2-distance coloring if any two vertices at distance at most $2$ from each other receive different colors. Let $G$ be a planar graph with girth at least $5$.
We prove that $G$ admits a $2$-distance coloring with $\D+4$ colors if $\D\geq 22$. 


\end{abstract}
\maketitle

\vspace*{-1.5em}

\section{Introduction}

All graphs in this paper are assumed to be simple, i.e., finite and undirected, with no loops or multiple edges. We refer to \cite{west} for terminology and notation not defined here. Let $G$ be a graph, we use $V(G),E(G),F(G),\D(G)$ and $g(G)$ to denote the vertex, edge and face set, the maximum degree and girth of $G$, respectively. If there is no confusion in the context, we abbreviate $\D(G),g(G)$ to $\D,g$. 
A 2-distance coloring is a vertex coloring where two vertices that are adjacent or have a common neighbour receive different colors, and the smallest number of colors for which $G$ admits a 2-distance coloring is known as the 2-distance chromatic number $\chi_2(G)$ of $G$.  A detailed survey on 2-distance coloring and related types of coloring is provided by Cranston \cite{daniel}. 

In 1977, Wegner \cite{wegner} posed the following conjecture.

\begin{conj}\label{conj:main}
For every planar graph $G$, $\chi_2(G) \leq 7$ if $\Delta=3$, $\chi_2(G) \leq \Delta+5$ if $4\leq \Delta\leq 7$, and $\chi_2(G) \leq  \floor[\big]{\frac{3\Delta}{2}}+1$ if $\Delta\geq 8$, 
\end{conj}

Wegner also constructed a graph to show that the upper bound for the case $\D\geq 8$ is indeed the best possible. While the conjecture remains widely open, there have been some partial solution. Thomassen \cite{thomassen} (independently by Hartke et al. \cite{hartke}) proved the conjecture for planar graphs with $\Delta = 3$. Moreover, there are upper bounds for $2$-distance chromatic number of planar graphs. For instance, van den Heuvel and McGuinness \cite{van-den} showed that $\chi_2(G) \leq 2\Delta + 25$, and Molloy and Salavatipour \cite{molloy}  established the bound $\chi_2(G) \leq \ceil[\big]{ \frac{5\D}{3}}+78$. \medskip


For planar graphs with girth restrictions, La and Montassier \cite{la-mont-2022} provided a concise overview of the most recent findings. For instance, Bu and Zhu \cite{bu-zu} showed that $\chi_2(G) \leq \Delta + 5$ if $G$ is a planar graph with $g \geq  6$, which confirms the Conjecture \ref{conj:main} for the planar graphs with girth six. This result was further refined to $\chi_2(G) \leq \Delta + 4$ when $\Delta\geq 6$ in \cite{deniz-g6}. Additionally, La \cite{la-2021} proved that $\chi_2(G) \leq \Delta + 3$ under two conditions: either $g \geq 7$ and $\Delta \geq 6$, or $g \geq 8$ and $\Delta \geq 4$.

When $G$ is a planar graph with a girth of at least $5$, Dong and Lin \cite{dong} proved that $\chi_2(G) \leq \Delta + 8$. This result was later improved in \cite{dong-lin-2017} to $\chi_2(G) \leq \Delta + 7$ when $\D\notin \{7,8\}$. Recently, Deniz \cite{deniz-g5} extended this results as $\chi_2(G) \leq \Delta + 7$ for every $\D$. Dong and Lin \cite{dong-lin-2017} also showed that $\chi_2(G) \leq \Delta + 4$ if $\D\geq 150$. More recently, Jin and Miao \cite{jin-miao-2022} proved that $\chi_2(G) \leq \Delta + 4$ if $\D\geq 40$, which is a corollary of a $2$-distance list coloring. \medskip


We focus on $2$-distance coloring of the planar graphs with girth at least $5$, and we improve a result of Dong and Lin \cite{dong-lin-2017} and partially extend the result of Jin and Miao \cite{jin-miao-2022}  as follows. 



\begin{thm}\label{thm:main}
If $G$ is a planar graph with  $g\geq 5$ and $\Delta\geq 22$, then $\chi_2(G) \leq \Delta + 4$. 
\end{thm}


Given a planar  graph $G$, 
we denote by $\ell(f)$ the length of a face $f$ and by $d(v)$ the degree of a vertex $v$. 
A $k$-vertex is a vertex of degree $k$. A $k^{-}$-vertex is a vertex of degree at most $k$, and a $k^{+}$-vertex is a vertex of degree at least $k$.  A $k$ ($k^-$ or $k^+$)-face is defined similarly. A vertex $u\in N(v)$ is called $k$-neighbour (resp. $k^-$-neighbour, $k^+$-neighbour) of $v$ if $d(u)=k$ (resp. $d(u)\leq k$, $d(u)\geq k$). A $k(d)$-vertex is a $k$-vertex adjacent to $d$ $2$-vertices. 

For a vertex $v\in V(G)$, we use $n_i(v)$ (resp. $n_2^3(v)$) to denote the number of $i$-vertices (resp. $2$-vertices having a $3$-neighbour) adjacent to $v$. Let $v\in V(G)$, we define $D(v)=\Sigma_{v_i\in N(v)}d(v_i)$. For $u,v\in V(G)$, we denote by $d(u,v)$ the distance between $u$ and $v$. 

For a path $uvw$ with $d(v)=2$, we say that $u$ and $v$ are \emph{weak-adjacent}. In particular, a pair of weak-adjacent vertices are called \emph{weak neighbour} of each other. 


\section{The Proof of Theorem \ref{thm:main}} \label{sec:premECE}

\subsection{The Structure of Minimum Counterexample} \label{sub:premECE}~~\medskip

Let $G'$ be a counterexample to Theorem \ref{thm:main}, i.e., $\D(G')=\D\geq 22$, $g(G')\geq 5$ and $\chi_2(G)> \D+4$, and let $G$ be a graph with minimum $|V(G)\cup E(G)|$ such that $\D(G)\leq \D$, $g(G)\geq g(G')$ and $\chi_2(G)> \D+4$. The set of graphs with these properties is non-empty, as $G'$ satisfies all of them. 
We will prove that $G$ does not exist, which contradicts the assumption that $G'$ exists.
Notice that, by minimality, $G-x$ has a 2-distance coloring with $\D+4$ colors for every $x\in V(G)\cup E(G)$. Obviously,  $G$ is connected and $\delta(G)\geq 2$. \medskip



We begin with introducing the concept of light and heavy vertices. A vertex $v$ is called \emph{light} if $D(v)<\D+4+n_2^3(v)$.  Denote by $n^l(v)$ the number of light vertices adjacent to $v$.
If $v$ is a vertex with $D(v)\geq \D+4+n^l(v)$, then it is called \emph{heavy}. It is worth noting that  every $2$-vertex adjacent to a $3$-vertex is a light vertex, so $n_2^3(v)\leq n^l(v)$ for any vertex $v$.

Let us now outline some structural properties that $G$ must carry, which will be  in use in the sequel. 



\begin{prop}\label{prop:non-heavy}
Let $v$ be a vertex in $G$ with $D(v)< \D+4+n^l(v)$. If $G$ has a 2-distance coloring $f$ with $\D+4$ colors such that $v$ is an uncolored vertex, then $v$ always has an available color. 
\end{prop}
\begin{proof}
Let $v$ be a vertex in $G$ with $D(v)< \D+4+n^l(v)$. Denote by $S$ the set of light vertices that are adjacent to $v$, and denote by  $R$ the set of all  $2$-neighbours of $v$ that are  adjacent to a $3$-vertex. Obviously,  we have $R\subseteq S$, also $n_2^3(v)=|R|$ and $n^l(v)=|S|$.  Consider a 2-distance coloring $f$ with $\D+4$ colors such that $v$ is an uncolored vertex, we decolor all vertices in $S$. Since $v$ has an available color as $D(v)< \D+4+n^l(v)$, we give an available color to it. On the other hand, we can easily recolor every vertex in $S$ because each of them is a light vertex. Indeed, since we decolor all vertices in $R$, each vertex in $S\setminus R$ has an available color. Therefore, we first recolor all vertices in $S\setminus R$, and then the vertices in $R$ with available colors.  Thus, we obtain a proper 2-distance coloring of $G$ with $\D+4$ colors.  
\end{proof}


Consider 2-distance coloring of $G$ with $\D+4$ colors, where a vertex $v$ is left uncoloured. One can easily infer from Proposition \ref{prop:non-heavy} that if $v$ is not heavy, then it always has an available color. We therefore have the following.

\begin{cor}\label{cor:light-ext}
If $v$ is a light vertex, then any 2-distance coloring $f$ with $\D+4$ colors such that $v$ is an uncolored vertex can always be extended to the whole graph.
\end{cor}

We will now prove that every neighbour of a light vertex is a heavy vertex. This result will be used further to derive some forbidden configurations.


\begin{prop}\label{prop:light-heavy}
If $v$ is a light vertex in $G$, then each neighbour of $v$ is a heavy vertex.
\end{prop}
\begin{proof}
Let $v$ be a light vertex. Assume to the contrary that there exists $u\in N(v)$ such that it is not heavy, i.e., $D(u)< \D+4+n^l(u)$. By minimality,  $G-uv$ has a 2-distance coloring $f$ with $\D+4$ colors. Denote by $S$ the set of light vertices that are adjacent to $u$. Obviously, $v\in S$ and $|S|=n^l(u)$.  We first decolor all vertices in $S\cup \{u\}$. Now, $u$ has at most $\D+3$ forbidden colors since $D(u)< \D+4+n^l(u)$, so we recolor $u$ with an available color. On the other hand, each vertex in $S$  has an available color by Corollary \ref{cor:light-ext}, since they are light vertices. Therefore, we give an available color to each vertex of $S$ so that we obtain a proper 2-distance coloring with  $\D+4$ colors, a contradiction.
\end{proof}

The following can be easily obtained from Proposition \ref{prop:light-heavy}.

\begin{cor}\label{cor:properties}
\begin{itemize}
\item[$(a)$] $G$ has no adjacent $2$-vertices.
\item[$(b)$] A $3$-vertex cannot have two $2$-vertices, i.e.,  $G$ has no $3(k)$-vertex for $k\in \{2,3\}$.
\item[$(c)$] A $4$-vertex cannot have four $2$-vertices, i.e.,  $G$ has no $4(4)$-vertex.
\end{itemize}
\end{cor}

\begin{lem}\label{lem:3-vertex-properties} 
Let $v$ be a $3$-vertex.
\begin{itemize}
\item[$(a)$] If $v$ has a $2$-neighbour, then $v$ is a heavy vertex. 
\item[$(b)$] If $v$ is a light $3(0)$-vertex having a $3$-neighbour $x$, then $x$ has a $12^+$-neighbour.
\item[$(c)$] If $v$ is a $3(1)$-vertex, then it cannot be adjacent to any $3(1)$-vertex.
\end{itemize}
\end{lem}
\begin{proof}
$(a)$. The claim follows from Proposition \ref{prop:light-heavy} together with the fact that every $2$ vertex having a $3$-neighbour is a light vertex.\medskip

$(b)$. Suppose that $v$ is a light $3(0)$-vertex, and let $x$ be a $3$-neighbour of $v$. Denote by $y,z$  the other neighbour of $x$. By  Proposition \ref{prop:light-heavy}, $x$ must be heavy, and so $D(x)\geq \D+4+n^l(x)$. Then $d(y)+d(z)\geq 24$. This infer that $x$ has  a $12^+$-neighbour. \medskip

$(c)$. Assume to the contrary that $v$ is adjacent to a $3(1)$-vertex $u$.  Let $x$ be the $2$-neighbour of $v$.  By minimality,  $G-vx$ has a 2-distance coloring with $\D+4$ colors.  We first decolor all $2$-vertices adjacent to some $3$-vertices. Then $v$ has at most $\D+3$ forbidden colors since $u$ has a decolored $2$-neighbour, so we give an available color to $v$. We next recolor all $2$-vertices adjacent to some $3$-vertices by Corollary \ref{cor:light-ext}. We therefore obtain a proper 2-distance coloring of $G$ with $\D+4$ colors, a contradiction. 
\end{proof}

%

\begin{lem}\label{lem:4-vertex-properties}
Let $v$ be a $4$-vertex.
\begin{itemize}
\item[$(a)$] If $v$ has a $2$-neighbour, then $v$ is a heavy vertex.
\item[$(b)$] If $v$ has two $2$-neighbours $x_1,x_2$ and a $3(1)$-neighbour $x_3$, then each $x_i$ is a heavy vertex. In particular, $v$ has a $20^+$-neighbour.
\item[$(c)$] If $v$ has three $2$-neighbours $x_1,x_2,x_3$, then each $x_i$ is a heavy vertex.
\end{itemize}
\end{lem}
\begin{proof}
$(a)$. 
Let $u$ be a $2$-neighbour of $v$. Assume to the contrary that $v$ is not heavy, i.e., $D(v)< \D+4+n^l(v)$. Consider a 2-distance $\D+4$ coloring of  $G-uv$, we decolor $u$ and $v$. Clearly $u$ has an available color since  $D(u)\leq \D+4$ and $v$ is uncolored, so we recolor it. Now, the only remaining uncolored vertex is $v$, and recall that $D(v)< \D+4+n^l(v)$. It then follows from Proposition \ref{prop:non-heavy} that we can extend this coloring to the whole graph, a contradiction. 


$(b)$. Suppose that $v$ has two $2$-neighbours $x_1,x_2$ and a $3(1)$-neighbour $x_3$. By $(a)$, $v$ is heavy vertex. Assume for a contradiction that one of $x_1,x_2$ is a non-heavy, say $x_1$. So we have $D(x_1)< \D+4+n^l(x_1)$. Consider a 2-distance $\D+4$ coloring of  $G-vx_2$, we decolor $v,x_1,x_2,x_3$ and the $2$-neighbour of $x_3$. Clearly $v$ currently has at most $\D+3$ forbidden colors, so we give an available color to it. Next, we recolor $x_3,x_2,x_1$ (in this order) by Proposition \ref{prop:non-heavy}, and the $2$-neighbour of $x_3$ by Corollary \ref{cor:light-ext}. So we obtain a proper $2$-distance $\D+4$ coloring of $G$, a contradiction. Thus, both $x_1$ and $x_2$ are heavy vertices. Similarly it can be shown that $x_3$ is heavy. 
On the other hand, let $z$ be the neighbour of $v$ other than $x_1,x_2,x_3$, and suppose that $z$ is a $19^-$-vertex. Similarly as above, consider a 2-distance $\D+4$ coloring of  $G-vx_1$, we decolor $v,x_1$ and the $2$-neighbour of $x_3$. Clearly $x_1$  currently has at most $\D+3$ forbidden colors, so we give an available color to it. Next, we recolor $v$ with an available color since it has at most $25$ forbidden colors, and finally the $2$-neighbour of $x_3$ by Corollary \ref{cor:light-ext}. So we obtain a proper $2$-distance $\D+4$ coloring of $G$, a contradiction.



$(c)$. It has a similar proof with $(b)$. 
\end{proof}



\begin{lem} \label{lem:5-12}
Let $v$ be a $k$-vertex with $5\leq k\leq 12$, and let $S$ be the set of $2$-neighbours of $v$.
If $n_2(v)=k-1$, and  $S$ has $r$ light vertices for $1\leq r \leq k-1$, then $v$ has a neighbour of degree at least $\D+6+r-2k$.
\end{lem}
\begin{proof}

Let  $n_2(v)=k-1$. This implies that $v$ has only one $3^+$-neighbour, say $z$.  Since $v$ has some light $2$-neighbours, $v$ must be heavy by Proposition \ref{prop:light-heavy}, and so $D(v)\geq  \D+4+n^l(v)\geq \D+4+r$. We therefore conclude that $z$ has at least $\D+4+r-(2k-2)$ neighbours. 
\end{proof}

%
%

We call a path $xyz$ as a \emph{poor path} if $5\leq d(y) \leq 6$ and $d(x)=d(z)=2$. If a poor path $xyz$ lies on the boundary of a face $f$, we call the vertex $y$ as \emph{f-poor vertex}.

It follows from the definition of the poor path that we can bound the number of those paths in a face, where we recall that $G$ has no adjacent $2$-vertices by Corollary \ref{cor:properties}(a).

\begin{cor}\label{cor:poor-path}
Let $f$ be a face having a poor path $xyz$ on its boundary. Then $f$ has at most  $ \floor[\big]{\frac{\ell(f)}{2}}$  $f$-poor vertices.  In particular, if $y$ is weak adjacent to two $7^+$-vertices lying on $ f$, then  $f$ has at most  $\floor[\big]{ \frac{\ell(f)-3}{2}}$  $f$-poor vertices. 
\end{cor}



%


In the rest of the paper, we will apply discharging to show that $G$ does not exist. We assign to each vertex $v$ a charge $\mu(v)=\frac{3d(v)}{2}-5$ and to each face $f$ a charge $\mu(f)=\ell(f)-5$. By Euler's formula, we have
$$\sum_{v\in V}\left(\frac{3d(v)}{2}-5\right)+\sum_{f\in F}(\ell(f)-5)=-10$$
Notice that only $2$ and $3$-vertices have negative initial charge. Since the sum of charges is negative, $G$ must have some $2$ and $3$-vertices.

We next present some rules and redistribute accordingly. Once the discharging finishes, we check the final charge $\mu^*(v)$ and $\mu^*(f)$. If $\mu^*(v)\geq 0$ and $\mu^*(f)\geq 0$, we get a contradiction that no such a counterexample can exist.


\subsection{Discharging Rules} \label{sub:} ~\medskip

We apply the following discharging rules. 

\begin{itemize}

\item[\textbf{R1:}] Every $2$-vertex receives $1$  from its each $9^-$-neighbour.
\item[\textbf{R2:}] Every light $3(0)$-vertex receives $\frac{1}{4}$  from its each $8^-$-neighbour.
\item[\textbf{R3:}] Let $v$ be a $3(1)$-vertex.
\begin{itemize}
\item[$(a)$] If $v$ has a neighbour $u$ with $3\leq d(u) \leq 5$, then $v$ receives $\frac{1}{4}$  from $u$.
\item[$(b)$] If $v$ has a neighbour $u$ with $6\leq d(u) \leq 8$, then $v$ receives $\frac{1}{2}$  from $u$.
\end{itemize}

\item[\textbf{R4:}] Let $v$ be weak adjacent to a $5(4)$-vertex $u$ such that $v$ has two $10^+$-neighbours. If $3 \leq d(v) \leq 9$, then $v$ gives $\frac{1}{4}$  to $u$.

\item[\textbf{R5:}] Let $v$ be weak adjacent to a $5(4)$-vertex $u$ such that $v$ has at most one $10^+$-neighbour. 
\begin{itemize}
\item[$(a)$]  If $8 \leq d(v) \leq 9$, then $v$ gives $\frac{1}{8}$  to $u$.
\item[$(b)$]  If $5 \leq d(v) \leq 7$ and  $v$ has two $3^+$-neighbours, then $v$ gives $\frac{1}{8}$  to $u$.
\item[$(c)$]  If $ d(v)=4$ and $n_2(v)=1$, then $v$ gives $\frac{1}{8}$  to $u$.
\item[$(d)$]  If $ d(v)=4$, $n_2(v)=2$, and $v$ has two $9^+$-neighbours, then $v$ gives $\frac{1}{8}$  to $u$.
\end{itemize}

\item[\textbf{R6:}] Every $9$-vertex gives $\frac{1}{2}$  to its each $k$-neighbour for $3\leq k \leq 8$.
\item[\textbf{R7:}] Every $10^+$-vertex transfers its positive charge equally to all of its neighbours. In particular, if a $11^+$-vertex $v$ has a $2$-neighbour $u$, then $v$ gives $1$  to $u$, and 
 $\frac{3d(v)-10}{2d(v)}-1$  to the other neighbour of $u$, instead of sending all of $\frac{3d(v)-10}{2d(v)}$ to $u$.
\item[\textbf{R8:}] If two $10^+$-vertices $u,v$ are adjacent, then each of $u,v$ gives $\frac{1}{2}$
  to the faces containing $uv$.
\item[\textbf{R9:}] Every face $f$ transfers its positive charge equally to  its incident $f$-poor vertices.
\end{itemize}

\noindent
\textbf{Checking} $\mu^*(v), \mu^*(f)\geq 0$, for $v\in V(G), f\in F(G)$\medskip

Clearly $\mu^*(f)\geq 0$ for each $f\in F(G)$, since every face transfers its positive charge equally to all of its incident $f$-poor vertices by R9. \medskip

We pick a vertex $v\in V(G)$ with $d(v)=k$. \medskip

\textbf{(1).} Let $k=2$. Then $\mu(v)=-2$. By Corollary \ref{cor:properties}-(a), $v$ is adjacent to two $3^+$-vertices. Clearly, $v$ receives $1$  from its each neighbour by R1 and R7, so  we have  $\mu^*(v)\geq 0$.  \medskip 


\textbf{(2).} Let $k=3$. Then $\mu(v)=-\frac{1}{2}$. By Corollary \ref{cor:properties}-(b), $v$ has at most one $2$-neighbour. 

First assume that $v$ has no $2$-neighbour. 
If $v$ has a $10^+$-neighbour $u$, then $v$ receives at least $1$  from $u$ by R7, and sends at most $\frac{1}{4}$  to each of the other neighbours by R2 and R3(a). So $\mu^*(v)\geq 0$. 
If $v$ has a $3$-neighbour and has no $10^+$-neighbour, then $v$ is a light vertex, since $D(v)\leq 3+9+9$. It then follows from Proposition \ref{prop:light-heavy} that all neighbours of $v$ are heavy. In this manner,  $v$ is a light $3(0)$-vertex, and so it receives $\frac{1}{4}$  from its each neighbour by R2 and R6. Thus, $\mu^*(v)= -\frac{1}{2} +3\times \frac{1}{4} >0$. 
If $v$ has neither $3$-neighbour nor $10^+$-neighbour, then $v$ is a light vertex or has two $9$-neighbours. In both cases, $v$ receives totally at least $\frac{1}{2}$  from its neighbours  by R2 and R6. So $\mu^*(v)\geq 0$.

Let us now assume that $v$ has a $2$-neighbour $x_1$, and so $v$ is a $3(1)$-vertex by Corollary \ref{cor:properties}-(b). Let $y_1$ be the other neighbour of $x_1$ except for $v$.  Denote by $w$ and $z$ the other neighbour of $v$ with $d(w)\leq d(z)$. Clearly, $x_1$ is a light vertex. By Proposition \ref{prop:light-heavy}, $v$ must be heavy vertex, and so $d(w)+d(z)\geq \D+3 \geq 25$. 
\begin{itemize}
\item If $w$ is a $3$-vertex, then $z$ would be a $\D$-vertex. Note that $w$ is a  $3(0)$-vertex by Lemma \ref{lem:3-vertex-properties}-(c). Additionally,  $w$ is not light vertex since otherwise $v$ could not be heavy. Then, $v$ receives at least $\frac{5}{4}$  from $z$ by R7, and $\frac{1}{4}$  from $w$ by R3(a). Therefore, $\mu^*(v)\geq 0$  after $v$ transfers $1$  to $x_1$ by R1.
\item If $4\leq d(w) \leq 5$, then $z$ is a $20^+$-vertex. 
Thus, $v$ receives at least $\frac{5}{4}$  from $z$ by R7, and $\frac{1}{4}$  from $w$ by R3(a). Therefore, $\mu^*(v)\geq 0$  after $v$ transfers $1$  to $x_1$ by R1. 

\item If $6\leq d(w) \leq 9$, then $z$ is a $16^+$-vertex. 
Thus, $v$ receives at least $1$  from $z$ by R7, and $\frac{1}{2}$  from $w$ by R3(b) and R6.
Therefore, $\mu^*(v)\geq 0$  after $v$ transfers $1$  to $x_1$ by R1. 
\item If $w$ is a $10^+$-vertex,  and so $z$ is a $10^+$-vertex as well, then $v$ receives at least $1$  from each of $w,z$ by R7. Therefore, $\mu^*(v)\geq \frac{1}{4}$  after $v$ transfers $1$  to $x_1$ by R1, and at most $\frac{1}{4}$  to $y_1$ by R4. \medskip  
\end{itemize}

\textbf{(3).} Let $k=4$.  Observe that $\mu^*(v)\geq 0$ by applying R2 and R3(a) when $v$ has no $2$-neighbour. We may therefore assume that $v$ has at least one $2$-neighbour. Recall that all neighbours of $v$ cannot be $2$-vertices by Corollary \ref{cor:properties}-(c). Thus $1\leq n_2(v)=t \leq 3$. Denote by $x_1,x_2,\ldots,x_t$ the $2$-neighbours of $v$, and let $y_i$ be the other neighbour of $x_i$, except for $v$. Note that $v$ is a heavy vertex by Lemma \ref{lem:4-vertex-properties}-(a), and so $D(v)\geq \D+4+n^l(v)$.\medskip

Let $n_2(v) = 1$. If $v$ has a $3$-neighbour $w$, then $v$ has also a $11^+$-neighbour $z$ since $v$ is heavy. Denote by $p$ the neighbour of $v$ other than $x_1,w,z$. By R7, $v$ receives at least $1$  from $z$. Thus have $\mu^*(v)\geq 0$ after $v$ sends $1$  to $x_1$ by R1 and at most $\frac{1}{4}$  to each of $w,p,y_1$ by R2, R3(a), R4 and R5(c). 
If $v$ has no $3$-neighbour, then we consider the vertex $y_1$. If $y_1$ is a $5$-vertex, then $x_1$ would be a light vertex, which implies that $v$ has a $9^+$-neighbour as $D(v)\geq \D+4+n^l(v)\geq 26$. 
Then, $v$ receives at least $\frac{1}{2}$ from its $9^+$-neighbour by R6-R7. Then, we have $\mu^*(v)\geq 0$ after $v$ sends $1$  to $x_1$ by R1 and at most $\frac{1}{4}$  to $y_1$ by R4 and R5(c). 
If $y_1$ is not a $5$-vertex, then we have $\mu^*(v)\geq 0$ after $v$ transfers $1$  to $x_1$ by R1.\medskip

Let $n_2(v) = 2$. Suppose that $v$ has a $3$-neighbour $y$. If $y$ is a $3(1)$-vertex, then, by Lemma \ref{lem:4-vertex-properties}-(b),  both $x_1 $ and $x_2$ are heavy vertices, and $v$ has a $20^+$-neighbour $z$. Thus $v$ receives at least $\frac{5}{4}$  from $z$ by R7, and sends $\frac{1}{4}$  to $y$ by R3(a).  Therefore, $\mu^*(v)\geq 0$  after $v$ transfers $1$  to each $x_i$ by R1. Recall that none of $y_i$'s is a $5$-vertex, because each $x_i$ is heavy.  
If $y$ is a light $3(0)$-vertex, then $v$ has a $20^+$-neighbour $z$, since $D(v)\geq \D+4+n^l(v)$ and $n^l(v)\geq 1$. Thus $v$ receives at least $\frac{5}{4}$  from $z$ by R7, and sends $\frac{1}{4}$  to $y$ by R2.  Notice that the rule R5(d) cannot be applied for $y_1,y_2$ since $v$ has no two $9^+$-neighbour.  Therefore, $\mu^*(v)\geq 0$  after $v$ transfers $1$  to each $x_i$ by R1. Besides, if $y$ is a $3(0)$-vertex that is not light, then $v$ has a $19^+$-neighbour $z$ since $v$ is heavy. Thus $v$ receives more than $1$  from $z$ by R7, and sends $1$  to each $x_i$ by R1. Note that the rule R5(d) cannot be applied for $y_1,y_2$, since $v$ has no two $9^+$-neighbours.  Hence, $\mu^*(v)\geq 0$.

We further suppose that $v$ has no $3$-neighbour. Denote by $y,z$ the $4^+$-neighbours of $v$. Assume that $r$ of $y_i$'s are $5(4)$-vertices for $0\leq r\leq 2$. Observe that $v$ has at least one $10^+$-vertex since $v$ is heavy. By R7, $v$ receives at least $1$ from its $10^+$-neighbour. It then follows that $\mu^*(v)\geq 0$ by applying R1 when $r=0$. So we may assume that $r\geq 1$. 
If $y,z$ are $10^+$-vertices,  then $v$ receives at least $1$  from each $y,z$ by R7, and sends $1$  to each $x_i$ by R1. Also, $v$ sends at most $\frac{1}{4}$  to each $y_i$ by R4. Obviously $\mu^*(v)> 0$.
If $y$ is $8^-$-vertex, then $z$ would be a $15^+$-vertex, since $v$ is heavy, and one of $x_1,x_2$ is light. Observe that $v$ does not need to give a charge to $y_i$ by R5(d), so $\mu^*(v)\geq 0$ after $v$ transfers $1$  to each $x_i$ by R1.  
If $y$ is $9^+$-vertex and $z$ is $10^+$-vertex, then $v$ receives at least $1$  from $z$ by R7, and at least $\frac{1}{2}$  from $y$ by R6-R7. Thus  $\mu^*(v)\geq 0$ after $v$ transfers $1$  to each $x_i$ by R1, and at most $\frac{1}{4}$  to each $y_i$ by R4 and R5(d). \medskip

Let $n_2(v) = 3$. Since $v$ is heavy, it has a $20^+$-neighbour $z$. Also, each $y_i$ is a $\D$-vertex  by Lemma \ref{lem:4-vertex-properties}-(c). It then follows that $v$ receives at least $\frac{5}{4}$  from $z$ by R7, and $\frac{1}{4}$  from each $y_i$ for $i\in [3]$ by R7. Therefore, $\mu^*(v)\geq 0$  after $v$ transfers $1$  to each $x_i$ by R1.  \medskip

\textbf{(4).} Let $k=5$. Observe that $\mu^*(v)\geq 0$ if $v$ has at most one $2$-neighbour by R1, R2, R3(a), R4 and R5(b). So, we may assume that $2\leq n_2(v)=t \leq 5$. Denote by $x_1,x_2,\ldots,x_t$ the $2$-neighbours of $v$, and for each $i\in [t]$, let $y_i$ be the other neighbour of $x_i$, except for $v$.\medskip

Let $n_2(v) = 2$. We first note that if $v$ has two $10^+$-neighbours $x,y$, then $v$ receives at least $1$  from each of them by R7, and so $\mu^*(v)> 0$  after $v$ transfers $1$  to each $x_i$ by R1; $\frac{1}{4}$  to its each $3$-neighbour by R2 and R3(a) (if exists);  at most $\frac{1}{4}$  to each $y_i$ by R4. Then, we may further assume that $v$ has no two $10^+$-neighbours. If $v$ has no $3$-neighbour, then $\mu^*(v)\geq 0$ after $v$ sends $1$  to its each $x_i$ by R1, and   $\frac{1}{8}$  to each $y_i$ by R5(b). We may therefore assume that $v$ has at least one $3$-neighbour. 
 If $v$ has one $3$-neighbour and two $4^+$-neighbours, then we again have $\mu^*(v)\geq 0$ after $v$ sends $1$  to each $x_i$ by R1, $\frac{1}{4}$  to its $3$-neighbour by R2, R3(a), and $\frac{1}{8}$  to each $y_i$ by R5(b). 
On the other hand, if $v$ has two $3$-neighbours and one $4^+$-neighbour, then, either $v$ has a $10^+$-neighbour or each $x_i$ is heavy by Proposition \ref{prop:light-heavy}. In the former, $v$ receives at least $1$  from its $10^+$-neighbour by R7. In the latter, each $x_i$ has a $\D$-neighbour, and so $v$ receives at least $\frac{1}{4}$  from each $y_i$ by R7.  In both cases, we have  $\mu^*(v)> 0$ by applying R1, R2, R3(a) and R5(b). We next assume that $v$ has three $3$-neighbours, say $z_1,z_2,z_3$. Obviously, $v$ is a light vertex, and so each $y_i$ is a $\D$-vertex by Proposition \ref{prop:light-heavy}.  It follows from applying R7 that $v$ receives at least $\frac{1}{4}$  from each $y_i$.   Therefore, $\mu^*(v)\geq  \frac{1}{4}$  after $v$ transfers $1$  to each $x_i$ by R1, and $\frac{1}{4}$  to each $z_i$ by R2 and R3(a).  \medskip

Let $n_2(v) = 3$. Let $z_1,z_2$ be $3^+$-neighbours of $v$. Notice first that if $z_1,z_2$ are $10^+$-vertices, then $v$ receives at least $1$  from each of them by R7, and so $\mu^*(v)\geq 0$  after $v$ transfers $1$  to each $x_i$ by R1, and  $\frac{1}{4}$  to each $y_i$ by R4. We may then further assume that $v$ has no two $10^+$-neighbours. Suppose first that $v$ has a $10^+$-neighbour, say $z_i$. 
If $z_{3-i}$ is a $3$-vertex, then either $z_i$ is $17^+$-vertex or each $x_i$ is heavy by Proposition \ref{prop:light-heavy}. In the former, $v$ receives at least $\frac{9}{8}$  from $z_i$ by R7, and sends  $1$  to each $x_i$ by R1,  $\frac{1}{4}$  to $z_{3-i}$ by R2, R3(a),  and $\frac{1}{8}$  to each $y_i$ by R5(b); so $\mu^*(v)\geq 0$. In the later, none of $x_i$'s has a $5$-neighbour, and so the rule R5 cannot be applied for $v$. It then follows that $v$ receives at least $1$  from $z_i$ by R7, and sends  $1$  to each $x_i$ by R1 and  $\frac{1}{4}$  to $z_{3-i}$ by R2, R3(a); so $\mu^*(v)\geq 0$. 
On the other hand, if $4 \leq d(z_{3-i})\leq 9$, then we similarly have $\mu^*(v)\geq 0$ by applying R7, R1 and R5(b), since $v$ does not need to give a charge  to $z_{3-i}$.
We now assume that $v$ has no $10^+$-neighbour, i.e., each neighbour of $v$ is a $9^-$-vertex. This forces that $v$ is a light vertex, and so each $y_i$ is a $\D$-vertex by Proposition \ref{prop:light-heavy}.
Since $v$ has three $2$-neighbours and two $3^+$-neighbours, there exists a face $f$ containing consecutive two $2$-neighbours of $v$. We may assume without loss of generality that $f$ is a face incident to $y_1,x_1,v,x_2,y_2$.  If $f$ is a $5$-face, i.e., $y_1$ is adjacent to $y_2$, then each of $y_1,y_2$ gives $\frac{1}{2}$  to the faces containing $y_1y_2$ by R8, so $f$ gets totally $1$ charge from $y_1,y_2$. By R9, $f$ transfer its charge to its unique $f$-poor vertex $v$. On the other hand, if $f$ is a $6^+$-face, then $v$ again gets at least $1$ charge  from $f$ by applying R9 together with Corollary \ref{cor:poor-path}. So the current charge of $v$ is at least $\mu(v)+1=\frac{7}{2}$. The vertex $v$ can give $\frac{1}{4}$  to each $z_i$ by R2, R3(a) when $z_i$ is a $3$-vertex. Thus, $\mu^*(v) \geq  0$  after $v$ transfers $1$  to each $x_i$ by R1.  \medskip


Let $n_2(v) = 4$. Suppose that $x_1,x_2,x_3,x_4$ have a clockwise order on the plane. Let $w$ be the $3^+$-neighbour of $v$.\\
First suppose that  $v$ is a light vertex, i.e., $w$ is a vertex of degree at most $\D-5$. So, each $x_i$ is a heavy vertex by Proposition \ref{prop:light-heavy}. This implies that each $y_i$ is a $\D$-vertex. Since $v$ has four $2$-neighbours and one $3^+$-neighbour, there exists a face $f$ containing consecutive two $2$-neighbours of $v$. We may assume without loss of generality that $f$ is a face incident to $y_1,x_1,v,x_2,y_2$. Similarly as above, if $f$ is a $5$-face, i.e., $y_1$ is adjacent to $y_2$, then each of $y_1,y_2$ gives $\frac{1}{2}$  to the faces containing $y_1y_2$ by R8, so $f$ gets totally $1$ charge from $y_1,y_2$, and sends it to $v$ by R9. On the other hand, if $f$ is a $6^+$-face, then $v$ again gets at least $1$ charge from $f$ by applying R8 and R9 together with Corollary \ref{cor:poor-path}. Moreover, each $y_i$ gives $\frac{3}{11}$  to $v$ by R7. So the current charge of $v$ is $\mu(v)+1+\frac{12}{11}=\frac{101}{22}$. The vertex $v$ can give $\frac{1}{4}$  to $w$ by R2, R3(a) when $w$ is a $3$-vertex. Thus, $\mu^*(v) \geq \frac{15}{44}$  after $v$ transfers $1$  to each $x_i$ by R1. \\
Now we suppose that  $v$ is a heavy vertex, i.e., $w$ is a vertex of degree at least $\D-4\geq 18$. By applying R7, $v$ gets at least  $\frac{11}{9}$  from $w$. Remark that if $r$ of $x_i$'s are light with $0\leq r \leq 4$, then $w$ is  a $(\D-4+r)^+$-vertex by Lemma \ref{lem:5-12}. We analysis all cases as follows:
 
If two of $x_i$'s are not light vertices, say $x_1,x_2$, then each of $y_1,y_2$ is a $(\D-1)^+$-vertex. Similarly as above,  each of $y_1,y_2$ gives at least $\frac{1}{4}$  to $v$ by R7.  So the current charge of $v$ is at least $\mu(v)+\frac{11}{9}+\frac{1}{2}=\frac{38}{9}$.  Thus, $\mu^*(v) \geq \frac{2}{9}$  after $v$ transfers $1$  to each $x_i$ by R1, where we note that the rule R5(b) cannot be applied for $y_3,y_4$ since $v$ has no two $3^+$-neighbours.
If two of $x_i$'s are light vertices, say $x_1,x_2$, and one of $x_i$'s is not light vertex, say $x_3$, then $w$ is a $20^+$-vertex. Also, $y_3$ is a $21^+$-vertex since $x_3$ is not light. Similarly as above, $y_3$ gives at least $\frac{1}{4}$  to $v$ by R7. In particular, $v$ gets at least  $\frac{5}{4}$  from $w$ by R7.  So the current charge of $v$ is at least $\mu(v)+\frac{5}{4}+\frac{1}{4}=4$.  Thus, $\mu^*(v)\geq 0 $  after $v$ transfers $1$  to each $x_i$ by R1, where we again note that the rule R5(b) cannot be applied for $y_1,y_2,\ldots, y_4$ since $v$ has no two $3^+$-neighbours. 
Finally, suppose that all $x_i$'s are light vertices. So, $w$ is a $\D$-vertex.  If one of $y_i$'s is a $19^+$-vertex, then $v$ gets totally at least $\frac{3}{2}$  from $w$ and $y_i$ by R7. Similarly as above, we have $\mu^*(v)\geq 0$. We therefore suppose that  each $y_i$ is a $18^-$-vertex for $i\in [4]$. In particular, for each $i\in [4]$,  $y_i$ is heavy since $x_i$ is light by Proposition \ref{prop:light-heavy}.
Recall that all $x_i$'s are light vertices, and $w$ is a $\D$-vertex. By R7, $v$ receives $\frac{14}{11}$ from $w$, so the current charge of $v$ is $\frac{83}{22}$. This means that $v$ needs at least $\frac{5}{22}$ charge in order to send $1$  to each $x_i$, and we will show that $v$ receives totally at least $\frac{1}{4}$ charge from $y_1,y_2,\ldots, y_4$ in the sequel. 

Consider $y_i$ and $y_{i+1}$ for $i\in [3]$, if they are not adjacent, then there exists a $6^+$-face incident to $y_i,x_i,v,x_{i+1},y_{i+1}$. By applying R9 together with Corollary \ref{cor:poor-path}, $f$ sends at least $\frac{1}{3}$  to $v$. Thus we further assume that $y_iy_{i+1}\in E(G)$ for each $i\in [3]$.  

If one of  $y_2,y_3$ is $3$-vertex, say $y_2$, then $y_2$ has two $5^+$-neighbours since each $y_i$ is a heavy, $x_i$ is light and $d(y_i)\leq 18$ for $i\in [4]$. In particular, if $y_2$ has two $10^+$-neighbour, then $y_2$ sends $\frac{1}{4}$  to $v$ by R4. Thus assume that $y_2$ has no two $10^+$-neighbours. Then $y_2$ has a $16^+$-neighbour $y_p$, and a $9^-$-neighbour $y_q$ with $d(y_q)\geq 5$,  for $p,q\in \{1,3\}$. Note that $y_q$ has at least two $3^+$-neighbours since it is heavy vertex and $y_2$ is a $3$-neighbour of $y_q$. It then follows that  $v$ gets at least $\frac{3}{16}$  from $y_p$ by R7, and at least $\frac{1}{8}$  from $y_q$ by R4, R5(a)-(b), so  $v$ gets totally at least $\frac{1}{4}$ charge. Thus we further assume that both  $y_2$ and $y_3$ are $4^+$-vertices.

Suppose first that $y_2$ and $y_3$ are $4$-vertices. If $y_2$ has another $2$-neighbour other than $x_2$, then $y_2$ has also a $19^+$-neighbour $z$ since $y_2$ is heavy and $x_2$ is light. We deduce that $z=y_1$ since $d(y_3)=4$, however it contradicts with the assuming $d(y_i)\leq 18$ for each $i\in [4]$. 
Thus $n_2(y_2)=1$. 
It then follows that $y_2$ sends at least $\frac{1}{8}$  to $v$ by R4,R5(c). By symmetry, $v$ receives at least $\frac{1}{8}$  from $y_3$, so  $v$ gets totally at least $\frac{1}{4}$ charge.

Suppose now that $y_2$ is $4$-vertex, and $y_3$ is a vertex with $5\leq d(y_3) \leq 9$. Observe that $y_3$ has two $3^+$-neighbours, so $y_3$ sends $\frac{1}{8}$  to $v$ by R5(a)-(b). We will next show that $v$ receives totally at least $\frac{1}{8}$ charge from $y_1,y_2$.
If $n_2(y_2)=1$, then $y_2$ sends at least $\frac{1}{8}$  to $v$ by R4, R5(c). 
We then suppose that $n_2(y_2)=2$. Recall that $y_2$ cannot have three $2$-neighbours since it has at least two $3^+$-neighbours $y_1,y_3$. 
If $y_2$ has two $9^+$-neighbours, then $y_2$ sends at least $\frac{1}{8}$  to $v$ by R4, R5(d).  If  $y_2$ has no two $9^+$-neighbours, then the rule R5 cannot be applied for $y_2$. 
In this case, however, we have $5\leq d(y_3) \leq 8$  and $14\leq d(y_1)$ since $y_2$ is heavy, $x_2$ is light, $5\leq d(y_3) \leq 9$ and $d(y_i)\leq 18$ for each $i\in [4]$. Then $y_1$ gives at least $\frac{1}{7}$ to $v$ by R7, so $v$ gets totally at least $\frac{1}{4}$ charge.


We finally suppose that $y_2$ and $y_3$ are $5^+$-vertex. If $d(y_2), d(y_3)\notin \{10,11,12,13\}$, then each of $y_2,y_3$ gives at least $\frac{1}{8}$  to $v$ by R4, R5(a)-(b), R7. 
Note that if both $y_2$ and $y_3$ are $10^+$-vertices, then we are done by R8. We may assume without lose of generality that $d(y_2)\in \{10,11,12,13\}$ and $5 \leq d(y_3) \leq 9$. By R5(a)-(b), $y_3$ gives $\frac{1}{8}$  to $v$.  So we only consider the vertex $y_2$, and we will show that $y_1$ gives at least $\frac{1}{8}$   to $v$ instead of $y_2$. Notice first that if $y_1$ and $y_2$ are $10^+$-vertices, then we are done by R8. So assume that $d(y_1)\leq 9$. If $8 \leq d(y_1) \leq 9$, then $y_1$ gives at least $\frac{1}{8}$   to $v$ by R4, R5(a). If $5 \leq d(y_1) \leq 7$, then $y_1$ would have an $3^+$-neighbour other than $y_2$ since $y_1$ is heavy,  so $y_1$ gives at least $\frac{1}{8}$   to $v$ by R5(b). If $d(y_1) = 4$ and $n_2(y_1)=1$, then  $y_1$ gives $\frac{1}{8}$  to $v$ by R5(c).   
If $d(y_1) = 4$ and $n_2(y_1)=2$, then  $y_1$ would have two $9^+$-neighbours, so $y_1$ gives $\frac{1}{8}$  to $v$ by R5(d). Recall that the case $d(y_1) = 4$ and $n_2(y_1)>2$ is not possible since $y_1$ is heavy and it is adjacent to $y_2$ with $d(y_2)\in \{10,11,12,13\}$. 
If $d(y_1)=3$, then $y_1$ would have two $10^+$-neighbours, so $y_1$ gives at least $\frac{1}{4}$   to $v$ by R4.  Consequently, $v$ receives totally at least $\frac{1}{4}$ charge from $y_1,y_2,y_3$ . Thus, $\mu^*(v)\geq 0$  after $v$ transfers $1$  to each $x_i$ by R1. \medskip


Let $n_2(v) = 5$. Suppose that $x_1,x_2,\ldots,x_5$ have a clockwise order on the plane. 
Obviously  $v$ is a light vertex. So, each $x_i$ is a heavy vertex by Proposition \ref{prop:light-heavy}. This implies that each $y_i$ is a $\D$-vertex, so each $y_i$ gives $\frac{3}{11}$  to $v$ by R7. Let $f_i$ be a face incident to $y_i,x_i,v,x_{i+1},y_{i+1}$ for $i=1,2$. If $f_1$ is a $5$-face, i.e., $y_1$ is adjacent to $y_2$, then each of $y_1,y_2$ gives $\frac{1}{2}$  to the faces containing $y_1y_2$ by R8, so $f_1$ gets totally $1$ charge, and transfers its charge to $v$ by applying R9 together with Corollary \ref{cor:poor-path}. On the other hand, if $f_1$ is a $6^+$-face, then $v$ again receives at least $1$  from $f_1$ by applying R9 together with Corollary \ref{cor:poor-path}. Similarly, $v$ receives at least $1$  from $f_2$ as well. So the current charge of $v$ is $\mu(v)+2+\frac{15}{11}=\frac{129}{22}$. Thus, $\mu^*(v)\geq \frac{19}{22}$  after $v$ transfers $1$  to each $x_i$ by R1. \medskip

\textbf{(5).} Let $k=6 $. Notice first that if all neighbours of $v$ are $3^-$-vertex, then $v$ is a light vertex, so $v$ cannot be weak adjacent to any $5$-vertex by Proposition \ref{prop:light-heavy}, which implies that the rules R4, R5 cannot be applied for $v$. Thus, we conclude that if $v$ has at most two $2$-neighbours, then $\mu^*(v)\geq 0$ by R1-R5, R7.  We may therefore assume that $3\leq n_2(v)=t \leq 6$. Denote by $x_1,x_2,\ldots,x_t$ the $2$-neighbours of $v$, and let $y_i$ be the other neighbour of $x_i$ except for $v$. \medskip

Let $n_2(v)=3$.  If $v$ has two $10^+$-neighbours, then $v$ receives at least $1$  from each of them by R7, and so $\mu^*(v)\geq 0$  after $v$ transfers $1$  to each $x_i$ by R1,  at most $\frac{1}{2}$  to each $3$-neighbour by R2, R3 and at most $\frac{1}{4}$  to each $y_i$ by R4. We may then assume that $v$ has no two $10^+$-neighbours.
If $v$ has at most one $3$-neighbour, then $\mu^*(v)\geq 0$ after  $v$ sends $1$  to each $x_i$ by R1, at most $\frac{1}{2}$  to the $3$-neighbour by R2, R3 and at most $\frac{1}{8}$  to each $y_i$ by R5(b).
Suppose next that $v$ has exactly two $3$-neighbours $z_1,z_2$. Denote by $z_3$ the last neighbour of $v$, if it is $10^+$-vertex, then $\mu^*(v)\geq 0$ similarly as above,  
since $v$ gets at least $1$  from $z_3$. Otherwise, if $z_3$ is  $9^-$-vertex, then $v$ would be a light vertex, and so each $x_i$ is heavy by Proposition \ref{prop:light-heavy}, i.e., each $y_i$ is a $20^+$-vertex. 
So, $\mu^*(v)> 0$ after  $v$ sends $1$  to each $x_i$ by R1 and at most $\frac{1}{2}$  to each $3$-neighbour by R2-R3.
 Now we assume that $v$ has three $3$-neighbours, so $v$ is a light vertex. Then each  $x_i$ is a heavy vertex by Proposition \ref{prop:light-heavy}, i.e., each $y_i$ is a $20^+$-vertex. By applying R7, $v$ gets totally at least $\frac{1}{2}$  from $y_1,y_2,y_3$ by R7. Thus $\mu^*(v)\geq 0$ after  $v$ sends $1$  to each $x_i$ by R1 and at most $\frac{1}{2}$  to each $3$-neighbour by R2-R3. \medskip

Let $n_2(v)=4$. Denote by $z_1,z_2$ the $3^+$-neighbours of $v$.   Clearly  $\mu^*(v) \geq 0$ when $v$ has  no $3$-neighbour and none of $y_i$ is a $5(4)$-vertex by R1.  If $v$ has a $3$-neighbour and a $9^-$-neighbour, then $v$ would be a light vertex. It follows from Proposition \ref{prop:light-heavy} that each $x_i$ is heavy, i.e., each $y_i$ is a $20^+$-vertex. By applying R7, $v$ gets totally at least $1$  from $y_1,y_2,y_3,y_4$. Thus $\mu^*(v) \geq 0$ after  $v$ sends $1$  to each $x_i$ by R1 and at most $\frac{1}{2}$  to each $3$-neighbour by R2-R3. 
If $v$ has a $3$-neighbour $z_1$ and a $10^+$-neighbour $z_2$, then $z_2$ gives $1$  to $v$ by R7, and so  $\mu^*(v) \geq 0$ after  $v$ sends $1$  to each $x_i$ by R1, at most $\frac{1}{2}$  to the $3$-neighbour by R2-R3 and at most $\frac{1}{8}$  to each $y_i$ by R5(b).  
We further assume that $z_1,z_2$ are $4^+$-vertices.  If one of them is $9^+$-vertex, say $z_1$, then $v$ receives at least $\frac{1}{2}$  from $z_1$ by R6, and so  $\mu^*(v) \geq 0$ by applying the same process as above. We may then assume that $z_1,z_2$ are $8^-$-vertices. Clearly $v$ is a light vertex. It follows from Proposition \ref{prop:light-heavy} that each $x_i$ is heavy, i.e., each $y_i$ is a $20^+$-vertex.  
Thus $\mu^*(v) \geq 0$ after  $v$ sends $1$  to each $x_i$ by R1. \medskip

Let $n_2(v)=5$. Denote by $z$ the $3^+$-neighbour of $v$.  If $z$ is a $10^+$-vertex, then it gives $1$  to $v$ by R7, and so  $\mu^*(v) \geq 0$ after  $v$ sends $1$  to each $x_i$ by R1. Note that the rule R5(b) cannot be applied since $v$ has no two $3^+$-neighbours. 
If $z$ is a $9^-$-neighbour, then $v$ would be a light vertex. It follows from Proposition \ref{prop:light-heavy} that each $x_i$ is heavy, i.e., each $y_i$ is a $20^+$-vertex. By applying R7, $v$ gets totally at least $\frac{5}{4}$  from $y_1,y_2,\ldots,y_5$. On the other hand, let $f$ be a face incident to $y_1,x_1,v,x_2,y_2$. If $f$ is a $5$-face, i.e., $y_1$ is adjacent to $y_2$, then each of $y_1,y_2$ gives $\frac{1}{2}$  to the faces containing $y_1y_2$ by R8, so $f$ gets totally $1$ charge. By applying R9 together with Corollary \ref{cor:poor-path}, $f$ transfer its charge to $v$. Besides, if $f$ is a $6^+$-face, then $v$ again gets at least $1$  from $f$ by applying R9 together with Corollary \ref{cor:poor-path}. Thus $\mu^*(v) \geq 0$ after  $v$ sends $1$  to each $x_i$ by R1 and at most $\frac{1}{2}$  to the $3$-neighbour by R2-R3 (if exists).  \medskip

Let $n_2(v)=6$. Since $v$ is a light vertex, each $x_i$ is heavy by Proposition \ref{prop:light-heavy}, and so each $y_i$ is a $20^+$-vertex. By applying R7, $v$ gets totally at least $\frac{3}{2}$  from $y_1,y_2,\ldots,y_6$.  Let $f$ be a face incident to $y_1,x_1,v,x_2,y_2$. Similarly as above, $v$ gets at least $1$  from $f$. So the current charge of $v$ is at least $\frac{11}{2}$. Thus, $\mu^*(v)> 0$  after $v$ transfers $1$  to each $x_i$ by R1. \medskip

\textbf{(6).} Let $k=7 $. Observe that $\mu^*(v)\geq 0$ if $v$ has at most three $2$-neighbours by R1-R5, R7. So, we may assume that $4\leq n_2(v)=t \leq 7$. Denote by $x_1,x_2,\ldots,x_t$ the $2$-neighbours of $v$, and let $y_i$ be the other neighbour of $x_i$ except for $v$. \medskip

Let $n_2(v)=4$. Notice first that if $v$ has two $10^+$-neighbours, then $v$ receives at least $1$  from each of them by R7, and so $\mu^*(v)\geq 0$  after $v$ transfers $1$  to each $x_i$ by R1,  at most $\frac{1}{2}$  to its each $3$-neighbour by R2-R3, and  $\frac{1}{4}$  to each $y_i$ by R4. We may then assume that $v$ has no two $10^+$-neighbours.  If $v$ has three $3$-neighbours, then $v$ is a light vertex, and so each $y_i$ is a $19^+$-vertex by Proposition \ref{prop:light-heavy}. It then follows that $\mu^*(v)\geq 0$ after  $v$ sends $1$  to each $x_i$ by R1, and at most $\frac{1}{2}$  to each $3$-neighbour by R1-R3. If $v$ has a $4^+$-neighbour, then $\mu^*(v)\geq 0$ after  $v$ sends $1$  to each $x_i$ by R1, and at most $\frac{1}{2}$  to each $3$-neighbour by R2-R3 and at most $\frac{1}{8}$  to each $y_i$ by R5(b). \medskip

Let $n_2(v)=5$. Denote by $z_1,z_2$ the $3^+$-neighbours of $v$. Similarly as above, if $v$ has two $10^+$-neighbours, then $\mu^*(v)\geq 0$. Assume that $v$ has no such two neighbours. 
Suppose that $z_1$ is a $3$-vertex. If $d(z_2) \leq 9$,  then $v$ is a light vertex, and so each $y_i$ is a $19^+$-vertex by Proposition \ref{prop:light-heavy}. By applying R7, $v$ gets totally at least $1$  from $y_1,y_2,\ldots,y_5$. Thus $\mu^*(v) \geq 0$ after  $v$ sends $1$  to each $x_i$ by R1 and at most $\frac{1}{2}$  to each $3$-neighbour by R2-R3.
If $d(z_2) \geq 10$, then $z_2$ gives $1$  to $v$ by R1, and so  $\mu^*(v) \geq 0$ by applying the same process as above. Suppose now that $z_1,z_2$ are $4^+$-neighbours. If $v$ is not weak adjacent to any $5(4)$-vertex, then  $\mu^*(v) \geq 0$ after  $v$ sends $1$  to each $x_i$ by R1. If $v$ is weak adjacent to a $5(4)$-vertex, then $v$ would have a $9^+$-neighbour, say $z_1$.  Thus $z_1$ gives at least $\frac{1}{2}$  to $v$ by R6-R7, and so  $\mu^*(v) \geq 0$ after  $v$ sends $1$  to each $x_i$ by R1 and at most $\frac{1}{8}$  to each $y_i$ by R5(b). \medskip


Let $n_2(v)=6$. If $v$ has a $10^+$-neighbour $z$,  then $z$ gives $1$  to $v$ by R1, and so  $\mu^*(v)\geq 0$ after  $v$ sends $1$  to each $x_i$ by R1. We may then assume that  $v$ has no $10^+$-neighbour. Then $v$ is a light vertex, and so each $y_i$ is a $19^+$-vertex by Proposition \ref{prop:light-heavy}. By applying R7, $v$ gets totally at least $1$  from $y_1,y_2,\ldots,y_6$. Thus $\mu^*(v) \geq 0$ after  $v$ sends $1$  to each $x_i$ by R1 and at most $\frac{1}{2}$  to its $3$-neighbour by R2-R3 (if exists). \medskip

Let $n_2(v)=7$. Since $v$ is a light vertex, each $x_i$ is  heavy, and so each $y_i$ is a $19^+$-vertex. By applying R7, $v$ gets totally at least $\frac{3}{2}$  from $y_1,y_2,\ldots,y_7$.  Thus, $\mu^*(v)> 0$  after $v$ transfers $1$  to each $x_i$ by R1. \medskip

\textbf{(7).} Let $k=8 $. Observe that $\mu^*(v)\geq 0$ if $v$ has at most five $2$-neighbours by R1-R5, R7. So, we may assume that $6\leq n_2(v)=t \leq 8$. Denote by $x_1,x_2,\ldots,x_t$ the $2$-neighbours of $v$, and let $y_i$ be the other neighbour of $x_i$ except for $v$.

Let $n_2(v)=6$. If $v$ has two $10^+$-neighbours $z_1,z_2$,  then each of them gives at least $1$  to $v$ by R7, and so  $\mu^*(v)\geq 0$ after  $v$ sends $1$  to each $x_i$ by R1 and at most $\frac{1}{4}$  to each $y_i$ by R4. We therefore assume that $v$ has no two $10^+$-neighbours. If $v$ has no $3$-neighbour, then  $\mu^*(v)\geq 0$ after  $v$ sends $1$  to each $x_i$ by R1, and  at most $\frac{1}{8}$  to each $y_i$ by R5(a). If $v$ has two $3$-neighbours, then $v$ would be a light vertex, i.e., $v$ is not weak adjacent to any $5$-vertex by Proposition \ref{prop:light-heavy}. Thus $\mu^*(v)\geq 0$ after  $v$ sends $1$  to each $x_i$ by R1, and at most $\frac{1}{2}$  to each $3$-neighbour by R2-R3. 
We further suppose that $v$ has  exactly one $3$-neighbour $z_1$ and one $4^+$-neighbour $z_2$. If $d(z_2)\leq 9$, then $v$ would be a light vertex, and so we have $\mu^*(v)\geq 0$ by applyign above process.  If $d(z_2)\geq 10$, then $z_2$ gives $1$  to $v$ by R7, and so  $\mu^*(v) \geq 0$ after  $v$ sends $1$  to each $x_i$ by R1,  at most $\frac{1}{2}$  to each $3$-neighbour by R2-R3, and  at most $\frac{1}{8}$  to each $y_i$ by R5(a). 

Let $n_2(v)=7$. If $v$ has a $10^+$-neighbour $z$,  then $z$ gives $1$  to $v$ by R7, and so  $\mu^*(v)\geq 0$ after  $v$ sends $1$  to each $x_i$ by R1 and $\frac{1}{8}$  to each $y_i$ by R5(a). Assume further that  $v$ has no $10^+$-neighbour. It then follows that $v$ is a light vertex, and so each $x_i$ is heavy  by Proposition \ref{prop:light-heavy}, i.e., each $y_i$ is a $18^+$-vertex. By applying R7, $v$ gets totally at least $1$  from $y_1,y_2,\ldots,y_7$. Thus $\mu^*(v) \geq 0$ after  $v$ sends $1$  to each $x_i$ by R1 and at most $\frac{1}{2}$  to its $3$-neighbour by R3-R4.

Let $n_2(v)=8$. Since $v$ is a light vertex, each $x_i$ is a heavy vertex by Proposition \ref{prop:light-heavy}, and so each $y_i$ is a $18^+$-vertex. By applying R7, $v$ gets totally at least $1$  from $y_1,y_2,\ldots,y_8$.  Thus, $\mu^*(v)> 0$  after $v$ transfers $1$  to each $x_i$ by R1. \medskip

\textbf{(8).} Let $k=9 $. Observe that $\mu^*(v)\geq 0$ if $v$ has at most six $2$-neighbours by R1-R7. So, we may assume that $7\leq n_2(v)=t \leq 9$. Denote by $x_1,x_2,\ldots,x_t$ the $2$-neighbours of $v$, and let $y_i$ be the other neighbour of $x_i$ except for $v$. Assume that $r$ of $y_i$'s are $5(4)$-vertices for $r\in \{1,2,\ldots,t\}$. 

Let $n_2(v)=7$.  Denote by $z_1,z_2$ the $3^+$-neighbours of $v$. If $z_1,z_2$ are $10^+$-neighbours,  then each of them gives $1$  to $v$ by R7, and so  $\mu^*(v)\geq 0$ after  $v$ sends $1$  to each $x_i$ by R1 and at most $\frac{1}{4}$  to each $y_i$ by R4. We therefore assume that $v$ has no two $10^+$-neighbours. If $r\leq 4$, then $\mu^*(v)\geq 0$ after  $v$ sends $1$  to each $x_i$ by R1, at most $\frac{1}{2}$  to  each of $z_1,z_2$  by R6, and  $\frac{1}{8}$  to at most four of $y_i$'s by R5(a). If $r\geq  5$, then $v$ has a $9^+$-neighbour, say $z_1$. Thus,  $\mu^*(v)\geq 0$ after  $v$ sends $1$  to each $x_i$ by R1, at most $\frac{1}{2}$  to $z_2$ by R6, and  $\frac{1}{8}$  to each $y_i$ by R5(a).


Let $n_2(v)=8$. Then $\mu^*(v)\geq 0$ after  $v$ sends $1$  to each $x_i$ by R1, at most $\frac{1}{2}$  to its $3^+$-neighbour by R6.


Let $n_2(v)=9$. Since $v$ is a light vertex, each $x_i$ is heavy by Proposiiton \ref{prop:light-heavy}, and so each $y_i$ is a $17^+$-vertex. By applying R7, $v$ gets totally at least $1$  from $y_1,y_2,\ldots,y_9$.  Thus, $\mu^*(v)> 0$  after $v$ transfers $1$  to each $x_i$ by R1. \medskip

\textbf{(9).} Let $k\geq 10 $. By R7, $v$ transfers its positive charge equally to its each neighbour.   
Note that if two $10^+$-vertices $x,y$ are adjacent, then they receive at least $1$ charge  from each other, and they transfers this charge into the faces containing $xy$ by R8. 
Consequently, $\mu^*(v)\geq 0$.



\begin{thebibliography}{99}

\bibitem{bu-zu}
Bu, Y., Zhu, X., \textit{An optimal square coloring of planar graphs},  Journal of Combinatorial Optimization \textbf{24}  580–592, (2012).

\bibitem{bu-zu-2}
Bu Y, Zhu J.  \textit{Channel Assignment with r-Dynamic Coloring}, 12th International Conference, AAIM
2018, Dallas, TX, USA, December 3-4, Proceedings pp 36-48,(2018)


\bibitem{daniel}
Cranston, D. W. \textit{Coloring, List Coloring, and Painting Squares of Graphs (and other related problems)},   The Electronic Journal of Combinatorics, \textbf{30}, 2, DS25, (2022).


\bibitem{deniz-g6}
Deniz, Z., \textit{An improved bound for 2-distance coloring of planar graphs with girth six.} arXiv preprint arXiv: 2212.03831, (2022).


\bibitem{deniz-g5}
Deniz, Z., \textit{Some results on 2-distance coloring of planar graphs with girth five.} arXiv preprint arXiv: 2308.00390, (2023).


\bibitem{dong}
Dong, W.  and Lin, W., \textit{An improved bound on 2-distance coloring plane graphs with girth 5}.  Journal of Combinatorial Optimization, 32(2), 645-655, (2016).

\bibitem{dong-lin-2017}
Dong, W., and Lin, W., \textit{On 2-distance coloring of plane graphs with girth 5.}  Discrete Applied Mathematics, 217, 495-505, (2017).

\bibitem{hartke}
Hartke, S. G.,    Jahanbekam, S. and  Thomas, B.,  \textit{The chromatic number of the square of subcubic planar graphs}. arXiv preprint arXiv:1604.06504, (2016). 

\bibitem{van-den}
Heuvel, J. van den,  McGuinness,  S., \textit{Coloring of the square of planar graph},  Journal of Graph Theory, 42  110–124, (2003).

\bibitem{jin-miao-2022}
Jin, Y. D., and Miao, L. Y., \textit{List 2-distance Coloring of Planar Graphs with Girth Five}.  Acta Mathematicae Applicatae Sinica, English Series, 38(3), 540-548, (2022).

\bibitem{la-2021}
La, H., \textit{2-distance list ($\Delta+3$)-coloring of sparse graphs}. arXiv:2105.01684,  (2021).

\bibitem{la-mont-2022}
La, H., and Montassier, M. \textit{2-Distance list $(\D+ 2)$-coloring of planar graphs with girth at least 10}. Journal of Combinatorial Optimization, 44(2), 1356-1375, 2022.



\bibitem{molloy}
 Molloy, M. and Salavatipour, M. R., \textit{ A bound on the chromatic number of the square
of a planar graph},  Journal of Combinatorial Theory Series B, 94,  189–213, (2005).

\bibitem{thomassen}
Thomassen, C., \textit{The square of a planar cubic graph is 7-colorable.} Journal of Combinatorial Theory Series B, 128:192–218, (2018).


\bibitem{wegner}
Wegner, G., \textit{Graphs with given diameter and a coloring problem}. Technical report, University of Dormund,
(1977).


\bibitem{west}
D. B. West. Introduction to graph theory, volume 2. Prentice hall Upper Saddle River, 2001.











\end{thebibliography}
\end{document}